\documentclass[12pt]{amsart}
\usepackage{amsmath}
\usepackage{amssymb}
\usepackage{amsthm}
\usepackage[numbers]{natbib}
\usepackage{enumerate}
\usepackage{dsfont}
\usepackage{a4wide}
\usepackage{color}
\usepackage{url}
\newcommand{\R}{{\mathbb R}}
\newcommand{\N}{{\mathbb N}}
\newcommand{\Z}{{\mathbb Z}}
\newcommand{\U}{{\mathbb U}}

\newcommand{\B}{{\mathbb B}}

\newcommand{\E}{\mathbb{E}}
\newcommand{\cE}{\mathcal{E}}

\newcommand{\Var}{\operatorname{Var}}
\newcommand{\Cov}{\operatorname{Cov}}

\newcommand{\dconv}{\Rightarrow}

\newcommand{\ind}{\mathds{1}}

\renewcommand{\P}{\mathbb{P}}

\newtheorem{theorem}{Theorem}
\newtheorem{proposition}{Proposition}
\newtheorem{lemma}{Lemma}

\theoremstyle{definition}

\newtheorem{remark}{Remark}

\def\alert#1{{\color{red}\fbox{TODO}}}
\newcommand{\comment}[1]{}
\def\limn{\lim_{n\to\infty}}

\def\indd#1{{\ind}_{\{#1\}}}

\def\pp#1{\left(#1\right)}

\def\bb#1{\left[#1\right]}

\def\summ#1#2#3{\sum_{#1=#2}^{#3}}
\def\mmid{\;\middle\vert\;}

\def\qmand{\quad\mbox{ and }\quad}
\newcommand{\eqnh}{\begin{eqnarray*}}
\newcommand{\eqne}{\end{eqnarray*}}
\newcommand{\eqnhn}{\begin{eqnarray}}
\newcommand{\eqnen}{\end{eqnarray}}
\newcommand{\equh}{\begin{equation}}
\newcommand{\eque}{\end{equation}}

\def\floor#1{\left\lfloor #1\right\rfloor}

\begin{document}\sloppy
\title[Infinite urn schemes]{From infinite urn schemes to decompositions of self-similar Gaussian processes}

\author{Olivier Durieu}
\address{
Olivier Durieu\\
Laboratoire de Math\'ematiques et Physique Th\'eorique, UMR-CNRS 7350\\
F\'ed\'eration Denis Poisson, FR-CNRS 2964\\
Universit\'e Fran\c{c}ois--Rabelais de Tours, Parc de Grandmont, 37200 Tours, France.
}
\email{olivier.durieu@lmpt.univ-tours.fr}

\author{Yizao Wang}
\address
{
Yizao Wang\\
Department of Mathematical Sciences\\
University of Cincinnati\\
2815 Commons Way\\
Cincinnati, OH, 45221-0025, USA.
}
\email{yizao.wang@uc.edu}

\date{\today}


\begin{abstract}
We investigate  a special case of infinite urn schemes first considered by \citet{Kar67},  especially its occupancy and odd-occupancy processes. We first propose a natural randomization of these two processes and their decompositions. We then establish functional central limit theorems, showing that each randomized process and its components converge jointly to a decomposition of certain self-similar Gaussian process. 
In particular, the randomized occupancy process and its components converge jointly to the decomposition of a time-changed Brownian motion $\mathbb B(t^\alpha), \alpha\in(0,1)$, and the randomized odd-occupancy process and its components converge jointly to a decomposition of fractional Brownian motion with Hurst index $H\in(0,1/2)$. The decomposition in the latter case is  a special case of the decompositions of bi-fractional Brownian motions recently investigated by \citet{lei09decomposition}. The randomized odd-occupancy process can also be viewed as correlated random walks, and in particular as a complement to the model recently introduced by \citet{HamShe13} as discrete analogues of fractional Brownian motions.
\end{abstract}
\maketitle

\section{Introduction}\label{sec:introduction}

We consider the classical infinite urn schemes, sometimes referred to as the  balls-in-boxes scheme. Namely, given a fixed infinite number of boxes, each time a label of the box is independently sampled according to certain probability $\mu$, and a ball is thrown into the corresponding box. This model has a very long history, dating back to at least \citet{Bah60}. For a recent survey from the probabilistic point of view, see \citet{GneHanPit07}. In particular, the sampling of the boxes forms naturally an exchangeable random partition of $\N$. Exchangeable random partitions have been extensively studied in the literature, and have connections to various areas in probability theory and related fields. See the nice monograph by \citet{pitman06combinatorial} for random partitions and more generally combinatorial stochastic processes. For various applications of the infinite urn schemes in biology, ecology, computational linguistics, among others, see for example \citet{BunFit93}.

In this paper, we are interested in a specific infinite urn scheme. 
More precisely, we consider 
$\mu$ as a probability measure on $\N:=\{1,2,\ldots\}$ which is regularly varying with index $1/\alpha$, $\alpha\in(0,1)$. See the definition in Section~\ref{sec:model}.
% and write $p_k:=\mu(\{k\})$ for all $k\ge1$. We assume $p_k$ to be regularly varying with index $1/\alpha$, $\alpha\in(0,1)$, among a few other technical assumptions. A detailed description of assumption on $(p_k)_{k\ge1}$ will be explicitly given later.
This model was first considered by \citet{Kar67} and we will refer to it by the {\it Karlin model} in the rest of the paper.  

We start by recalling the main results of \citet{Kar67}. Let $(Y_i)_{i\ge1}$ represents the independent sampling from $\mu$ for each round $i\ge1$, and 
\[
Y_{n,k} := \summ i1n \indd{Y_i=k}
\] 
be the total counts of sampling of the label $k$ in the first $n$ rounds, or equivalently how many balls thrown into the box $k$ in the first $n$ rounds.  In particular, Karlin investigated the asymptotics of two statistics: the total number of boxes that have been chosen in the first $n$ rounds, denoted by 
\[
Z^*(n) := \sum_{k\geq1}\indd{Y_{n,k}\neq 0},
\]
and the total number of boxes that have been chosen by an odd number of times in the first $n$ rounds, denoted by 
\[
U^*(n) := \sum_{k\ge1}\indd{Y_{n,k} \text{ is odd}}.
\]
The processes $Z^*$ and $U^*$ are referred to as the {\it occupancy process} and the {\it odd-occupancy process}, respectively.
While $Z^*$ is a natural statistics to consider in view of sampling different species, the investigation of $U^*$ is motivated via the following light-bulb-switching point of view from~\citet{spitzer64principles}. Each box $k$ may represent the status (on/off) of a light bulb, and each time when $k$ is sampled, the status of the corresponding light bulb is switched either from on to off or from off to on. Assume that all the light bulbs are off at the beginning. In this way, $U^*(n)$ represents the total number of light bulbs that are on at time $n$.
 
Central limit theorems have been established for both processes in \citep{Kar67}, in form of 
\equh\label{eq:karlin}
\frac{Z^*(n) - \E Z^*(n)}{\sigma_n}\Rightarrow {\mathcal N}(0,\sigma_Z^2)\qmand \frac{U^*(n) - \E U^*(n)}{\sigma_n}\Rightarrow {\mathcal N}(0,\sigma_U^2)
\eque
for some normalization $\sigma_n$, with $\sigma_Z^2$ and $\sigma_U^2$ explicitly given as the variances of the limiting normal distributions, and where $\Rightarrow$ denotes convergence in distribution. We remark that $\sigma_n^2$ is of the order $n^\alpha$, up to a slowly varying function at infinity.

The next seemingly obvious task is to establish the functional central limit theorems for the two statistics. However, to the best of our knowledge, this  has not been addressed in the literature.
Here by functional central limit theorems, or weak convergence, we are thinking of results in the form of (in terms of $Z^*$)
\equh\label{eq:karlinWIP}
\left\{\frac{Z^*(\floor{nt}) - \E Z^*(\floor{nt})}{\sigma_n}\right\}_{t\in[0,1]}\Rightarrow \{\mathbb Z^*(t)\}_{t\in[0,1]},
\eque
in space $D([0,1])$ for some normalization sequence $\sigma_n$ and a Gaussian process $\mathbb Z^*$.  In view of~\eqref{eq:karlin} and the fact that $\sigma_n^2$ has the same order as $n^\alpha$, the scaling limit $\mathbb Z^*$, if exists, is necessarily self-similar with index $\alpha/2$. 

In this paper, instead of addressing this question directly, we consider a more general framework by introducing the randomization to the Karlin model (see Section~\ref{sec:model} for the exact definitions). The randomization of the Karlin model reveals certain rich structure of the model. In particular, it has a natural decomposition. Take the randomized occupancy process $Z^\varepsilon$ for example. We will write
\[
Z^\varepsilon(n) = Z^\varepsilon_1(n)+Z^\varepsilon_2(n)
\]
and prove a joint weak convergence result in form of
\[
\frac1{\sigma_n}(Z^\varepsilon_1(\floor{nt}),Z^\varepsilon_2(\floor{nt}),Z^\varepsilon(\floor{nt}))_{t\in[0,1]}\dconv \left(\Z_1(t),\Z_2(t),\Z(t)\right)_{t\in[0,1]},
\]
in $D([0,1])^3$, such that 
\[
\Z = \Z_1+\Z_2 \quad\mbox{ with }\quad \mbox{ $\Z_1$ and $\Z_2$ independent}.
\] 
In other words, the limit trivariate Gaussian process $(\Z_1(t),\Z_2(t),\Z(t))_{t\in[0,1]}$ can be constructed by first considering two independent Gaussian processes $\Z_1$ and $\Z_2$ with covariance to be specified, and then set $\Z(t) := \Z_1(t)+\Z_2(t), t\in[0,1]$; in this way its finite-dimensional distributions are also determined. We refer such results as {\it weak convergence to the decomposition of a Gaussian process}. 
Similar results for the randomized odd-occupancy process are also obtained.
Here is a brief summary of the main results of the paper.
\begin{itemize}
\item 
As expected, various self-similar Gaussian processes appear in the limit. In this way, the randomized Karlin model and its components, including $Z^*$ and $U^*$ as special quenched cases, provide discrete counterparts of several self-similar Gaussian processes. These processes include notably the fractional Brownian motion with Hurst index $H = \alpha/2$, the bi-fractional Brownian motion with parameter $H = 1/2, K = \alpha$, and a new self-similar process $\Z_1$. 
\item Moreover, in view of the weak convergence to the decomposition, the randomized Karlin model are discrete counterparts of certain decompositions of self-similar Gaussian processes.
The randomized occupancy process and its two components converge weakly to a new decomposition of the time-changed Brownian motion $(\B(t^\alpha))_{t\ge0},\alpha\in(0,1)$ (Theorem~\ref{thm:Z}). The randomized odd-occupancy process and its two components converge weakly to a decomposition of the fractional Brownian motion with Hurst index $H =\alpha/2\in(0,1/2)$ (Theorem~\ref{thm:U}). This decomposition is a particular case of the decompositions of bi-fractional Brownian motion recently discovered by \citet{lei09decomposition}. 
\end{itemize}

Self-similar processes have been extensively studied in probability theory and related fields \citep{embrechts02selfsimilar}, often related to the notion of long-range dependence \citep{samorodnitsky06long,pipiras15long}.
Among the self-similar processes arising in the limit in this paper, the most widely studied one is the fractional Brownian motion. 
Fractional Brownian motions, as generalizations of Brownian motions, have been widely studied and used in various areas of probability theory and applications. These processes are the only centered Gaussian processes that are self-similar with stationary increments. %A fractional Brownian motion $(\B^H(t))_{t\ge0}$ with Hurst index $H\in(0,1)$ has covariance function
%\[\Cov(\B^H(s),\B^H(t)) = \frac12\pp{s^{2H}+t^{2H}-|s-t|^{2H}}, s,t\ge 0.\]
The investigation of fractional Brownian motions dates back to \citet{kolmogorov40wienersche} and \citet{mandelbrot68fractional}. As for limit theorems, there are already several models that converge to fractional Brownian motions in the literature. See \citet{davydov70invariance,taqqu75weak,Enr04,kluppelberg04fractional,peligrad08fractional,mikosch07scaling,HamShe13} for a few representative examples. A more detailed and extensive survey of various models can be found in \citet{pipiras15long}. Besides, we also obtain limit theorems for bi-fractional Brownian motions introduced by \citet{houdre03example}. They often show up in decompositions of self-similar Gaussian processes; see for example \citep{ChaTud09,lei09decomposition}. However, we do not find other discrete models in the literature.
As for limit theorems illustrating decompositions of Gaussian processes as ours do, in the literature we found very few examples; see Remark~\ref{rem:1}. 

Our results  connect the Karlin model, a discrete-time stochastic process, to several continuous-time self-similar Gaussian processes and their decompositions. By introducing new discrete counterparts, we hope to improve our understanding of these Gaussian processes. In particular,
the proposed randomized Karlin model can also be viewed as correlated random walks, in a sense complementing the recent model introduced by \citet{HamShe13} that scales to fractional Brownian motions with Hurst index $H\in(1/2,1)$. Here, the randomized odd-occupancy process ($U^\varepsilon$ below) is defined in a similar manner, and scales to fractional Brownian motions with $H\in(0,1/2)$.

The paper is organized as follows. Section~\ref{sec:prelim} introduces the model in details and present the main results. Section~\ref{sec:Poi} introduces and investigates the Poissonized models. The de-Poissonization is established in Section~\ref{sec:dP}.

 \section{Randomization of Karlin model and main results}\label{sec:prelim}
\subsection{Karlin model and its randomization}\label{sec:model}
We have introduced the original Karlin model in Section~\ref{sec:introduction}. Here, we specify the regular variation assumption. Recall the definition of $(p_k)_{k\ge 1}$. We assume that $p_k$ is non-increasing, and define the infinite counting measure $\nu$ on $[0,\infty)$ by
$$
\nu(A):=\sum_{j\ge 1}\delta_{\frac{1}{p_j}}(A)
$$
for any Borel set $A$ of $[0,\infty)$, where $\delta_x$ is the Dirac mass at $x$.
For all $t>1$, set
\begin{equation}\label{nu}
\nu(t):=\nu([0,t])=\max\{j\ge 1\mid p_j\ge1/t\},
\end{equation}
where $\max\emptyset =0$.
Following \citet{Kar67}, the main assumption is that $\nu(t)$ is a regularly varying function at $\infty$ with index $\alpha$ in $(0,1)$, that is for all $x> 0$,
$\lim_{t\to\infty}\nu(tx)/\nu(t)= x^\alpha$,
or equivalently
\begin{equation}\label{RV}
 \nu(t)=t^{\alpha}L(t),\; t\ge 0,
\end{equation}
where $L(t)$ is a slowly varying function as $t\to\infty$, i.e.\ for all $x>0$, $\lim_{t\to\infty}L(tx)/L(t)=1$. For the sake of simplicity, one can think of 
\begin{equation*}
 p_k\underset{k\to\infty}{\sim}Ck^{-\frac1\alpha} \text{ for some }\alpha\in(0,1)\text{ and a normalizing constant }C>0.
\end{equation*}
This implies $\nu(t)\underset{t\to\infty}{\sim}C^\alpha t^{\alpha}$.

We have introduced two random processes considered in \citet{Kar67}: the occupancy process and the odd-occupancy process as
\[
Z^*(n):=\sum_{k\ge1}\ind_{\{Y_{n,k}\ne0\}}\qmand U^*(n):=\sum_{k\ge1}\ind_{\{Y_{n,k}\text{ is odd}\}}
\]
respectively. To introduce the randomization, 
%More generally, other limit processes can be reached by introducing randomized occupancy processes as described below. In fact this randomization have two interests. On the one hand, it will help to understand the nature of the scaling limit  for the non-randomized processes by making a link with more familiar processes (as the fractional Brownian motion for the odd-occupancy process). On the other hand, these randomized processes will provide other simple discrete models that scale to these well known fractional Gaussian processes.
let $\varepsilon:=(\varepsilon_k)_{k\ge1}$ be a sequence of i.i.d.\ Rademacher random variables (i.e.\ $\P(\varepsilon_k=1)=\P(\varepsilon_k=-1)=1/2$) defined on the same probability space as the $(Y_n)_{n\ge1}$ and independent from them. In the sequel, we just say that $\varepsilon$ is a Rademacher sequence in this situation (and thus implicitly, $\varepsilon$ will always be independent of $(Y_n)_{n\ge1}$).

Let $\varepsilon$ be a Rademacher sequence. We introduce the {\it randomized occupancy process} and the {\it randomized odd-occupancy process} by 
$$
Z^\varepsilon(n):=\sum_{k\ge 1}\varepsilon_k\ind_{\{Y_{n,k}\ne0\}}\qmand  U^\varepsilon(n):=\sum_{k\ge 1}\varepsilon_k\ind_{\{Y_{n,k}\text{ is odd}\}}.
$$
We actually will work with decompositions of these two processes given by
\[
Z^\varepsilon(n) = Z^\varepsilon_1(n)+Z^\varepsilon_2(n)\qmand U^\varepsilon(n) = U^\varepsilon_1(n)+U^\varepsilon_2(n),
\]
where
\begin{align}
Z_1^\varepsilon(n)&:=\sum_{k\ge 1}\varepsilon_k\left(\ind_{\{Y_{n,k}\ne0\}}-p_k(n)\right)\mbox{ and }
Z_2^\varepsilon(n):=\sum_{k\ge 1}\varepsilon_kp_k(n),\quad n\ge 1,\label{eq:Z}
\\
U_1^\varepsilon(n)&:=\sum_{k\ge 1}\varepsilon_k\left(\ind_{\{Y_{n,k}\text{ is odd}\}}-q_k(n)\right)\mbox{ and }
U_2^\varepsilon(n):=\sum_{k\ge 1}\varepsilon_kq_k(n),\quad n\ge 1,\label{eq:U}
\end{align}
with for all $k\ge 1$ and $n\ge 1$,
\begin{align*}
p_k(n)&:=\P\left(Y_{n,k}\ne0\right)=1-(1-p_k)^n,\\
q_k(n)&:=\P\left(Y_{n,k}\text{ is odd}\right)=\frac12(1-(1-2p_k)^n).
\end{align*}

In the preceding definitions, the exponent $\varepsilon$ refers to the randomness given by the Rademacher sequence $(\varepsilon_k)_{k\ge 1}$. Nevertheless, in some of the following statements, the sequence of $(\varepsilon_k)_{k\ge 1}$ can be chosen fixed (deterministic) in $\{-1,1\}^{\N}$. Then the corresponding processes can be considered as ``quenched'' versions of the randomized process. For this purpose, it is natural to introduce the centering with $p_k(n)$ and $q_k(n)$ respectively above.  Actually, we will establish quenched weak convergence for $Z^\varepsilon_1$ and $U^\varepsilon_1$ (see Theorem~\ref{thm:1} and Remark~\ref{rem:quenched}). With a little abuse of language, for both cases we keep $\varepsilon$ in the notation and add an explanation like `{\it for a Rademacher sequence $\varepsilon$}' or `{\it for all fixed $\varepsilon\in\{-1,1\}^\N$}', respectively.

\subsection{Main results}
As mentioned in the introduction, we are interested in the scaling limits of the previously defined processes. We denote by $D([0,1])$ the Skorohod space of cadlag functions on $[0,1]$ with the Skorohod topology (see \citep{Bil99}). Throughout, we write 
\[
\sigma_n := n^{\alpha/2} L(n)^{1/2},
\]
where $\alpha$ and $L$ are the same as in the regular variation assumption~\eqref{RV}. Observe that $\nu(n) = L(n) =  \sigma_n = 0$ for $n<1/p_1$. Therefore, when writing $1/\sigma_n$ we always assume implicitly $n\geq 1/p_1$.
We obtain similar results for $Z^\varepsilon$ and $U^\varepsilon$. Below are the main results of this paper.

\begin{theorem}\label{thm:Z}
 For a Rademacher sequence $\varepsilon$,
$$
\frac{1}{\sigma_n}\left(Z_1^\varepsilon(\floor{nt}),Z_2^\varepsilon(\floor{nt}),Z^\varepsilon(\floor{nt})\right)_{t\in[0,1]}\dconv \left(\Z_1(t),\Z_2(t),\Z(t)\right)_{t\in[0,1]},
$$
in $(D([0,1]))^3$, where $\Z_1,\Z_2,\Z$ are centered Gaussian processes, such that 
\[
\Z=\Z_1+\Z_2,
\]
$\Z_1$ and $\Z_2$ are independent,  
 and they have covariances
\begin{align*}
\Cov(\Z_1(s),\Z_1(t)) & =\Gamma(1-\alpha)\left((s+t)^\alpha-\max(s, t)^\alpha\right), \\
\Cov(\Z_2(s),\Z_2(t)) & =\Gamma(1-\alpha)\left(s^\alpha+t^\alpha-(s+t)^\alpha \right), \; \\
\Cov(\Z(s),\Z(t)) & =\Gamma(1-\alpha)\min(s, t)^\alpha,\quad s,t\geq 0.
\end{align*}
\end{theorem}
% \comment{
% The process $\Z$ can be seen as a generalization of the Brownian motion for a different self-similarity index. The covariances of $\Z$ can be deduced from the ones of $\Z_1$ and $\Z_2$, but the direct computation explains the minimum that appears. By independence of the $\varepsilon_k$, we have
% $$
% \Cov(Z^\varepsilon(\floor{ns}),Z^\varepsilon(\floor{nt}))=\E\left(\sum_{k\ge 1}\ind_{\{Y_{\floor{ns},k}\ne 0\}}\ind_{\{Y_{\floor{nt},k}\ne 0\}}\right)=\sum_{k\ge1}p_k(\floor{ns}\wedge\floor{nt}).
% $$ 
% Then the asymptotic can be computed using \cite[Theorem 1']{Kar67}.
% }
%Again the process $\Z$ is very few studied in the literature. The reason is that the fractional Brownian motion is a more natural generalization as it preserves the stationarity of the increments, which is not the case for $\Z$. Nevertheless, our approach shows that $\Z$ is a natural process since it is the scaling limit of a simple correlated random walk as described before.

\begin{theorem}\label{thm:U}
 For a Rademacher sequence $\varepsilon$,
$$
\frac{1}{\sigma_n}\left(U_1^\varepsilon(\floor{nt}),U_2^\varepsilon(\floor{nt}),U^\varepsilon(\floor{nt})\right)_{t\in[0,1]}\dconv \left(\U_1(t),\U_2(t),\U(t)\right)_{t\in[0,1]},
$$
in $(D([0,1]))^3$, where $\U_1,\U_2,\U$ are centered Gaussian processes such that 
\[
\U=\U_1+\U_2,
\] 
$\U_1$ and $\U_2$ are independent,  
and they have covariances
\begin{align*}
\Cov(\U_1(s),\U_1(t))&=\Gamma(1-\alpha)2^{\alpha-2}\left((s+t)^\alpha-|t-s|^\alpha\right),\\
\Cov(\U_2(s),\U_2(t))&=\Gamma(1-\alpha)2^{\alpha-2}\left(s^\alpha+t^\alpha-(s+t)^\alpha \right),\\
\Cov(\U(s),\U(t))&=\Gamma(1-\alpha)2^{\alpha-2}\left(s^\alpha+t^\alpha-|t-s|^\alpha\right), \quad s,t\ge 0.
\end{align*}
\end{theorem}
To achieve these results, we will first prove the convergence of the first ($Z^\varepsilon_1$ and $U^\varepsilon_1$) and the second ($Z^\varepsilon_2$ and $U^\varepsilon_2$) components, respectively.
For the first components we have the following stronger result. 
\begin{theorem}\label{thm:1}
For all fixed $\varepsilon\in\{-1,1\}^{\N}$, 
$$
\left(\frac{Z_1^\varepsilon(\floor{nt})}{\sigma_n}\right)_{t\in[0,1]}\dconv
\left(\Z_1(t)\right)_{t\in[0,1]}
\qmand\left(\frac{U_1^\varepsilon(\floor{nt})}{\sigma_n}\right)_{t\in[0,1]}\dconv
\left(\U_1(t)\right)_{t\in[0,1]},
$$
in $D([0,1])$, where $\Z_1$ and $\U_1$ are as in Theorems~\ref{thm:Z} and~\ref{thm:U}.
\end{theorem}
\begin{remark}
\label{rem:quenched}
Theorem~\ref{thm:1} is a quenched functional central limit theorem. %See for example the notion of {\it almost surely weak convergence} recently introduced by \citet{grubel14functional}.
In particular, when taking $\varepsilon = \vec 1 = (1,1,\dots)\in\N$, Theorem~\ref{thm:1} recovers and generalizes the central limit theorems for $Z^*(n)$ and $U^*(n)$ established in \citet{Kar67} (formally stated in~\eqref{eq:karlin}): the (non-randomized) occupancy and odd-occupancy processes of the Karlin model scale to the continuous-time processes $\Z_1$ and $\U_1$, respectively.
Moreover, as the limits in Theorem~\ref{thm:1} do not depend on the value of $\varepsilon$, this implies the annealed functional central limit theorems (the same statement of Theorem~\ref{thm:1} remains true for a Rademacher sequence $\varepsilon$), and entails essentially the joint convergence to the decomposition. %A final subtle point: our result is slightly stronger than almost surely weak convergence, for which it suffices to establish {\it for almost all } $\varepsilon\in\{-1,1\}^\N$. 
%The same remark applies to $\tilde Z_1^\varepsilon,\tilde U_1^\varepsilon$ in Proposition~\ref{prop:Poisson1}.
\end{remark}

Now we take a closer look at the processes appearing in Theorem~\ref{thm:Z} and Theorem~\ref{thm:U} and the corresponding decompositions. 
The decomposition of $\U$ is a special case of the general decompositions established in \citet{lei09decomposition} for bi-fractional Brownian motions. Recall that a bi-fractional Brownian motion with parameter $H\in(0,1),K\in(0,1]$ is a centered Gaussian process with covariance function
\equh\label{eq:R}
R^{H,K}(s,t) = \frac1{2^K}\pp{\pp{t^{2H}+s^{2H}}^K - |t-s|^{2HK}}.
\eque
It is noticed in \citep{lei09decomposition} that one can write
\equh\label{eq:lei}
\frac1{2^K}\pp{t^{2HK}+s^{2HK} - |t-s|^{2HK}} = R^{H,K}(s,t) + \frac1{2^K}\pp{t^{2HK}+s^{2HK}-(t^{2H}+s^{2H})^{K}},
\eque
where  the left-hand side above is a multiple of the covariance function of a fractional Brownian motion with Hurst index $HK$, and  the second term on the right-hand side above is positive-definite and hence a covariance function. Therefore,~\eqref{eq:lei} induces a decomposition of a fractional Brownian motion with Hurst index $HK$ into a bi-fractional Brownian motion and another self-similar Gaussian process.

Comparing this to Theorem~\ref{thm:U}, we notice that our decomposition of $\U$ corresponds to the special case of~\eqref{eq:lei} with $H=1/2, K=\alpha$. Up to a multiplicative constant, $\U$ is a fractional Brownian motion with Hurst index $H = \alpha/2$. The process  $\U_1$ is the bi-fractional Brownian motion with $H = 1/2, K=\alpha$, and it is also known as the odd-part of the two-sided fractional Brownian motion; see \citet{DzhZan04}. That is
$$
\left(\U_1(t)\right)_{t\ge 0}\stackrel{fdd}{=}\sqrt{2^{\alpha}\Gamma(1-\alpha)}\left(\frac12(\B^{\alpha/2}(t)-\B^{\alpha/2}(-t))\right)_{t\ge 0},
$$
where $\B^{\alpha/2}$ is a two-sided fractional Brownian motion on $\R$ with Hurst index $\alpha/2\in(0,1)$. %Our result shows that a natural discrete counterpart of $\U_1$ is the odd-occupancy process of the Karlin model. 
The process $\U_2$ admits a representation
$$
\U_2(t)=2^{\alpha/2-1}\sqrt{\alpha}\int_0^\infty (1-e^{st})s^{-\frac{\alpha+1}{2}}d\B(s),\; t>0,
$$
where $(\B(t))_{t\in[0,1]}$ is the standard Brownian motion. 
It is shown that $\U_2(t)$ has a version with infinitely differentiable path for  $t\in (0,\infty)$ and absolutely continuous path for  $t\in [0,\infty)$. At the same time, $\U_2$ also appears in the decomposition of sub-fractional Brownian motions \citep{BojGorTal04,ChaTud09}.

For the decomposition of $\Z$ in Theorem~\ref{thm:Z},  to the best of our knowledge it is new in the literature. Remark that $\Z$ is simply a time-changed Brownian motion $(\Z(t))_{t\geq 0}\stackrel{fdd}=\Gamma(1-\alpha)(\B(t^\alpha))_{t\ge 0}$, and that $\Z_2\stackrel{fdd}= 2^{-\alpha/2+1}\U_2$. The latter is not surprising as 
 the coefficients $q_k(n)$ and $p_k(n)$ have the same asymptotic behavior. However, we cannot find related reference for $\Z_1$ in the literature. 
 The following remark on $\Z_1$ has its own interest.

\begin{remark} The process $\Z_1$ may be related to bi-fractional Brownian motions as follows. One can write
 \[
 (s^{1/\alpha}+t^{1/\alpha})^\alpha - |s-t| = 2\bb{\pp{s^{1/\alpha}+t^{1/\alpha}}^\alpha - \max(s,t)} + \bb{s + t - \pp{s^{1/\alpha}+t^{1/\alpha}}^\alpha}, s,t\geq 0.
 \]
That is, 
\[
(\mathbb V(t))_{t\ge0}\stackrel{fdd}=\pp{2\Z_1(t^{1/\alpha}) + \Z_2(t^{1/\alpha})}_{t\ge0},
\]
where $\Z_1$ and $\Z_2$ are as before and independent, and $\mathbb V$ is a centered Gaussian process with covariance 
\[
\Cov(\mathbb V(s),\mathbb V(t)) = \Gamma(1-\alpha)2^\alpha R^{1/(2\alpha),\alpha}(s,t).
\]
Therefore, as another consequence of our results, we have shown that for the bi-fractional Brownian motions, the covariance function $R^{H,K}$ in~\eqref{eq:R} is well defined for $H = 1/(2\alpha), K = \alpha$ for all $\alpha\in(0,1)$. The range $\alpha\in(0,1/2]$ is new. %When $\alpha\in(1/2,1)$, see \citet[Section 3]{russo06bifractional} for some properties on this sub-class of bi-fractional Brownian motions.
\end{remark}
%As for the proofs, we will establish the convergence for $\Z_1^\varepsilon, \Z_2^\varepsilon, \U_1^\varepsilon,\U_2^\varepsilon$ first. We actually will establish the weak convergence of $\Z_1^\varepsilon$ and $\U_1^\varepsilon$ in the almost sure (quenched) sense; see Remark~\ref{rem:quenched}. The joint weak convergence above then follows as an immediate consequence. 
To prove the convergence of each individual process, we apply the Poissonization technique. Each of the Poissonized  processes $\tilde\Z_1^\varepsilon,\tilde\Z_2^\varepsilon,\tilde\U_1^\varepsilon,\tilde\U_2^\varepsilon$ is  an infinite sum of independent random variables, of which the covariances are easy to calculate, and thus the finite-dimensional convergence follows immediately. The hard part for the Poissonized processes is to establish the tightness for $\tilde Z_1^\varepsilon$ and $\tilde U_1^\epsilon$. For this purpose we apply a chaining argument. Once the weak convergence for the Poissonized models are established, we couple the Poissonized models with the original ones and bound the difference. The second technical challenges lie in this de-Poissonization step. Remark that \citet{Kar67} also applied the Poissonization technique in his proofs. Since he only worked with central limit theorems and us the functional central limit theorems, our proofs are more involved.

\begin{remark}
One may prove the weak convergence  $(Z^\varepsilon(\floor{nt})/\sigma_n)_{t\in[0,1]}\Rightarrow(\Z(t))_{t\in[0,1]}$ and $(U^\varepsilon(\floor{nt})/\sigma_n)_{t\in[0,1]}\Rightarrow(\U(t))_{t\in[0,1]}$  directly, without using the decomposition. We do not present the proofs here as they do not provide insights on the decompositions of the limiting processes. 
\end{remark}
\begin{remark}\label{rem:1}
We are not aware of other limit theorems for the decomposition of processes in a similar manner as ours, but with two exceptions. One is the symmetrization well investigated in the literature of empirical processes \citep{VaaWel96}. Take for a simple example the empirical distribution function 
\[
\mathbb F_n(t) := \frac1n \summ i1n \indd {X_i\leq t}
\]
where $X_1,X_2,\dots$ are i.i.d.~with uniform $(0,1)$ distribution. By symmetrization one considers an independent Rademacher sequence $\varepsilon$ and
\[
\mathbb F_n^{\varepsilon}(t)  := \frac1n\summ i1n \varepsilon_i\indd{X_i\leq t},\quad \mathbb F_n^{\varepsilon,1}(t) := \frac1n\summ i1n\varepsilon_i\pp{\indd{X_i\leq t}-t}  \quad\mbox{ and }\quad \mathbb F_n^{\varepsilon,2}(t) := \frac1n\summ i1n \varepsilon_i t.
\]
It is straight-forward to establish 
\[
\sqrt n\,(\mathbb F_n^{\varepsilon}(t), \mathbb F_n^{\varepsilon,1}(t),\mathbb F_n^{\varepsilon,2}(t))_{t\in[0,1]}\Rightarrow(\B(t), \B(t)-t\B(1),t\B(1))_{t\in[0,1]}.
\]
This provides an interpretation of the definition of Brownian bridge via $\B^{bridge}(t) := \B(t) - t\B(1), t\in[0,1]$.

 The other example of limit theorems for decompositions is the recent paper by \citet{bojdecki12particle} who provided a particle-system point of view for the decomposition of fractional Brownian motions. The model considered there is very different from ours, and so is the decomposition in the limit. 
\end{remark}
\subsection{Correlated random walks}\label{sec:discussion}

We first focus our discussion on $U^\varepsilon$. One can interpret the process $U^\varepsilon$ as a correlated random walk by writing
\equh\label{eq:RW}
U^\varepsilon(n) = X_1+\cdots+X_n,
\eque 
for some random variables $(X_i)_{i\ge 1}$ taking values in $\{-1,1\}$ with equal probabilities, and the dependence among steps is determined by a random partition of $\N$. Viewing $X_i$ as the step of a random walk at time $i$, $U^\varepsilon$ in~\eqref{eq:RW} then represents a correlated random walk. To have such a representation, recall $(Y_i)_{i\ge1}$ and consider the random partition of $\N$  induced by the equivalence relation that $i\sim j$ if and only if $Y_i=Y_j$; that is, the integer $i$ and $j$ are in the same component of the partition if and only if  the $i$-th and $j$-th balls fall in the same box. 
Once $(Y_n)_{n\ge 1}$ is given and thus all components are determined, one can define $(X_n)_{n\ge 1}$ as follows.
For each $k\ge 1$, suppose all elements in component $k$ (defined as $\{i:Y_i = k\}$) are listed in increasing order $n_1<n_2<\cdots$, and set $X_{n_1}:=\varepsilon_k$ and iteratively $X_{n_{\ell+1}} := -X_{n_\ell}$. In this way, it is easy to see that each $X_n$ is taking values in $\{-1,1\}$ with equal probabilities, and conditioning on $(Y_n)_{n\ge1}$,  $X_i$ and $X_j$ are completely dependent if $i\sim j$, and independent otherwise. The verification of~\eqref{eq:RW} is straight-forward.

The above discussion describes how to construct a correlated random walk from random partitions in two steps. The first is to sample from a random partition. The second is to assign $\pm1$ values to $(X_n)_{n\ge 1}$ conditioned on the random partition sampled. 
A similar interpretation can be applied to another model of correlated random walks  introduced in \citet{HamShe13}. The Hammond--Sheffield model also constructed a collection of random variables taking values in $\{-1,1\}$, of which the dependence is determined by a random partition of $\Z$, in form of a random forest with infinitely many components indexed by $\Z$. There are two differences between the Hammond--Sheffield model and the randomized odd-occupancy process $U^\varepsilon$: first, the underlying random partition is different: notably, the random partition in the infinite urn scheme is exchangeable, while this is not the case for the random partition introduced in $\Z$ in \citep{HamShe13}; rather, the random partition there inherits certain long-range dependence property, which essentially determines that the Hurst index in the limit must be in $(1/2,1)$. Second, for $X_i$ in the same component of random partitions, Hammond--Sheffield model set them to take the {\it same value} (all $1$ or all $-1$ with equal probabilities), independently on each component.

The {\it alternative way} of assigning values for $X_i$ in the same component is the key idea in our framework. Clearly this has been considered by \citet{spitzer64principles} and \citet{Kar67}, if not earlier. Actually, \citet{HamShe13} suggested, as an open problem, to apply the alternative way of assigning values to their model and asked whether the modified model scales to fractional Brownian motions with Hurst index in $(0,1/2)$. In our point of view, in order to obtain a discrete model in the similar flavor of the Hammond--Sheffield model that scales to a fractional Brownian motion with Hurst index $H\in(0,1/2)$, the alternative way of assigning values is crucial, while the underlying random forest with long memory  is not that essential. Our results support this point of view. Actually, looking for a model in a similar spirit to complement the Hammond--Sheffield model as the discrete counterparts of fractional Brownian motions was another motivation for this paper. At the same time, the aforementioned suggestion in \citep{HamShe13} remains a challenging model to analyze. 

As for the occupancy process, similarly one can view $Z^\varepsilon$ as a correlated random walk. The random partition being the same, this time for each component $k$ with elements $n_1<n_2<\cdots$, we set $X_{n_1} = \epsilon_k, X_{n_i} = 0, i\ge 2$ to obtain
\[
Z^\varepsilon = X_1+\cdots+X_n. 
\]
The dependence of this random walk is simpler than the odd-occupancy process.

\section{Poissonization}\label{sec:Poi}
Recall that we are interested in the processes $Z^\varepsilon$ and $U^\varepsilon$, and instead to deal with them directly we work with the decompositions $Z^\varepsilon=Z_1^\varepsilon+Z_2^\varepsilon$ and $U^\varepsilon=U_1^\varepsilon+U_2^\varepsilon$ with the components defined in~\eqref{eq:Z} and~\eqref{eq:U}.

\subsection{Definitions and preliminary results}

The first step in the proofs is to consider the Poissonized versions of all the preceding processes in order to deal with sums of independent variables. Let $N$ be a Poisson process with intensity $1$, independent of the sequence $(Y_n)_{n\ge 1}$ and of the Rademacher sequence $\varepsilon$  considered before. We set 
\[
N_k(t):=\sum_{\ell=1}^{N(t)}\ind_{\{Y_\ell=k\}}, t\ge0, k\ge 1.
\]
Then the processes $N_k$, $k\ge 1$, are independent Poisson processes with respective intensity $p_k$.
We now consider the Poissonized processes, for all $t\ge 0$,
$$
\tilde Z^\varepsilon(t):= \sum_{k\ge 1}\varepsilon_k\ind_{\{N_k(t)\ne0\}} \qmand \tilde U^\varepsilon(t):=\sum_{k\ge 1}\varepsilon_k\ind_{\{N_k(t)\text{ is odd}\}}.
$$
These Poissonized randomized occupancy and odd-occupancy processes  have similar decompositions as the original processes
\[
\tilde Z^\varepsilon =\tilde Z_1^\varepsilon+\tilde Z_2^\varepsilon \qmand  \tilde U^\varepsilon =\tilde U_1^\varepsilon+\tilde U_2^\varepsilon
\]
with
\begin{align*}
\tilde Z_1^\varepsilon(t)& :=\sum_{k\ge 1}\varepsilon_k\left(\ind_{\{N_k(t)\ne0\}}-\tilde p_k(t)\right),\quad
\tilde Z_2^\varepsilon(t):=\sum_{k\ge 1}\varepsilon_k\tilde p_k(t),
\\
\tilde U_1^\varepsilon(t)&:=\sum_{k\ge 1}\varepsilon_k\pp{\ind_{\{N_k(t)\text{ is odd}\}}-\tilde q_k(t)}
,\quad
\tilde U_2^\varepsilon(t):=\sum_{k\ge 1}\varepsilon_k\tilde q_k(t),
\end{align*}
and
\begin{align*}
\tilde p_k(t) & :=\P(N_k(t)\ne0)=1-e^{-p_kt},\\
\tilde q_k(t) & :=\P(N_k(t)\text{ is odd})=\frac12(1-e^{-2p_kt}).
\end{align*}

Using the independence and the stationarity of the increments of Poisson processes, we derive the following useful identities.
For all $0\leq s\leq t$ and all $k\ge 1$,
\begin{align}\label{subadd:p}
0\leq\tilde p_k(t)-\tilde p_k(s)& =(1-\tilde p_k(s))\tilde p_k(t-s)  \le \tilde p_k(t-s),
\\
\label{subadd:q}
0\leq \tilde q_k(t)-\tilde q_k(s) & =(1-2\tilde q_k(s))\tilde q_k(t-s) \le \tilde q_k(t-s).
\end{align}
Note that, in particular, the functions $\tilde p_k$ and $\tilde q_k$ are sub-additive.
Further, we will have to deal with the asymptotics of the sums over $k$ of the $\tilde p_k$ or $\tilde q_k$. For this purpose, recall that (see \cite[Theorem 1]{Kar67}) the assumption \eqref{RV} implies
\begin{equation}\label{equiv}
V(t):=\sum_{k\ge 1}(1-e^{-p_kt})\sim \Gamma(1-\alpha)t^\alpha L(t),\;\text{ as }t\to\infty.
\end{equation}
We will need further estimates on the asymptotics of $V(t)$ as stated in the following lemma.

\begin{lemma}\label{lem:V}
For all $\gamma \in (0,\alpha)$, there exists a constant $C_{\gamma }>0$ such that 
$$
{V(nt)}\le C_{\gamma } t^{\gamma }\sigma_n^2, \;\text{ uniformly in }t\in[0,1], n\ge1.
$$
\end{lemma}
\begin{proof}
Recall the definition of the integer-valued function $\nu$ in \eqref{nu}. By integration by parts, we have for all $t>0$,
$$
V(t)=\int_0^\infty(1-e^{-t/x})d\nu(x)=\int_0^\infty x^{-2}e^{-1/x} \nu(tx) dx.
$$
Observe that $\nu(t) = 0$ if and only if $t\in[0,1/p_1)$ by definition, and in particular $L(t) = 0$ if and only if $t\in[0,1/p_1)$. Thus, 
\[
\frac{V(nt)}{\sigma_n^2} = \int_{1/(ntp_1)}^\infty x^{-2}e^{-1/x}\nu(ntx)dx  = t^\alpha \int_{1/(ntp_1)}^\infty x^{\alpha-2}e^{-1/x}\frac{L(ntx)}{L(n)}dx.  
\]
Now we introduce 
\[
L^*(t) = \left\{\begin{array}{l@{\mbox{ if }}l}L(1/p_1) & t\in[0,1/p_1)\\ L(t) & t\in[1/p_1,\infty)\end{array}\right.,
\]
and obtain
\[
\frac{V(nt)}{\sigma_n^2} \leq t^\alpha\int_0^\infty x^{\alpha-2}e^{-1/x}\frac{L^*(ntx)}{L^*(n)}dx.
\]
Let $\delta>0$ be such that $\alpha+\delta<1$ and $\alpha-\delta>\gamma$. Observe that $L^*$ has the same asymptotic behavior as $L$ by definition. In addition, $L^*$ is bounded away from $0$ and $\infty$ on any compact set of $[0,\infty)$. Thus,  by Potter's theorem (see \cite[Theorem 1.5.6]{BinGolTeu87}) there exists a constant $C_\delta>0$ such that for all $x,y>0$
$$
\frac{L^*(x)}{L^*(y)}\le C_\delta \max\left(\left(\frac xy\right)^{\delta}, \left(\frac xy\right)^{-\delta} \right).
$$
We infer, uniformly in $t\in[0,1]$, 
\begin{align*}
 \frac{V(nt)}{\sigma_n^2}
&\le  C_\delta  t^\alpha \int_0^\infty x^{\alpha-2} e^{-1/x} \max\left(\left(tx \right)^{\delta}, \left(tx\right)^{-\delta} \right) dx\\
&\le  C_\delta  t^{\alpha-\delta} \left( \int_0^1 x^{\alpha-\delta-2} e^{-1/x}dx +  \int_1^\infty x^{\alpha+\delta-2} e^{-1/x} dx\right),
\end{align*}
and both integrals are finite (the second one because we have taken $\delta$ such that $\alpha+\delta<1$). Further, $t^{\alpha-\delta}\le t^{\gamma }$ for all $t\in[0,1]$ and thus the lemma is proved.
\end{proof}

\subsection{Functional central limit theorems}
We now establish the invariance principles for the Poissonized processes.

\begin{proposition}\label{prop:Poisson1}
 For all fixed $\varepsilon\in\{-1,1\}^{\N}$, 
$$
\left(\frac{\tilde Z_1^\varepsilon(nt)}{\sigma_n}\right)_{t\in[0,1]}\dconv
\left(\Z_1(t)\right)_{t\in[0,1]}
\;\text{ and }\;
\left(\frac{\tilde U_1^\varepsilon(nt)}{\sigma_n}\right)_{t\in[0,1]}\dconv
\left(\U_1(t)\right)_{t\in[0,1]},
$$
in $D([0,1])$, where $\Z_1$ is as in  Theorem~\ref{thm:Z} and $\U_1$ is as in Theorem~\ref{thm:U}.
\end{proposition}
\begin{proof} In the sequel $\varepsilon\in\{-1,1\}^{\N}$ is fixed.
The proof is divided into three steps.\medskip

\noindent{\em (i) The covariances.}
Using the independence of the $N_k$, and that $\varepsilon_k^2=1$ for all $k\ge 1$, we infer that for all $0\leq s\leq t$,
\begin{multline*}
\Cov\left(\tilde Z_1^\varepsilon(ns),\tilde Z_1^\varepsilon(nt)\right)
=\sum_{k\ge 1}\left(\P(N_k(ns)\ne 0,\,N_k(nt)\ne 0)-\tilde p_k(ns)\tilde p_k(nt)\right)\\
=\sum_{k\ge 1}\left((1-e^{-p_kns})-(1-e^{-p_kns})(1-e^{-p_knt})\right)
=V(n(s+t))-V(nt),
\end{multline*}
whence by \eqref{equiv}, 
\[
\limn \frac{1}{\sigma_n^2}\Cov\left(\tilde Z_1^\varepsilon(ns),\tilde Z_1^\varepsilon(nt)\right)=\Gamma(1-\alpha)\left((s+t)^\alpha-t^\alpha\right).
\]
For the odd-occupancy process, using the independence  and the stationarity of the increments of the Poisson processes, for $0\le s \le t$,
\begin{align*}
\Cov&\left(\tilde U_1^\varepsilon(ns),\tilde U_1^\varepsilon(nt)\right)
 =  \sum_{k\ge 1}\left(\P(N_k(ns)\text{ is odd},\,N_k(nt)\text{ is odd})-\tilde q_k(ns)\tilde q_k(nt)\right)\\
& =  \sum_{k\ge 1}\left(\tilde q_k(ns)(1-\tilde q_k(n(t-s)))-\tilde q_k(ns)\tilde q_k(nt)\right)\\
& =  \frac14\sum_{k\ge 1}(1-e^{-2p_kns})(e^{-2p_kn(t-s)}+e^{-2p_knt}) = \frac14\left(V(2n(t+s))+V(2n(t-s))\right).
\end{align*}
Thus, again by \eqref{equiv},
\[
\limn\frac{1}{\sigma_n^2}\Cov\left(\tilde U_1^\varepsilon(ns),\tilde U_1^\varepsilon(nt)\right) 
=\Gamma(1-\alpha)2^{\alpha-2}\left((t+s)^\alpha-(t-s)^\alpha \right).
\]

\noindent{\em (ii) Finite-dimensional convergence.} 
The finite-dimensional convergence for both processes is a consequence of the Lindeberg central limit theorem, using 
the Cram\'er--Wold device. Indeed, for any choice of constants $a_1,\ldots,a_d\in\R$, $d\ge 1$, and any reals $t_1,\ldots,t_d\in[0,1]$, the independent random variables $\varepsilon_k\sum_{i=1}^da_i(\ind_{\{N_k(nt_i)\ne0\}}-\tilde p_k(nt_i))$, $k\ge 1$, $n\ge 1$ are uniformly bounded. This entails the finite-dimensional convergence for $(\tilde Z_1^\varepsilon(nt)/\sigma_n)_{t\in[0,1]}$. The proof for $(\tilde U_1^\varepsilon(nt)/\sigma_n)_{t\in[0,1]}$ is similar.

\medskip

\noindent{\em (iii)  Tightness.}
The proof of the tightness is technical and delayed to Section~\ref{sec:tightness}.
\end{proof}

\begin{proposition}\label{prop:Poisson2}
 For any Rademacher sequence $\varepsilon=(\varepsilon_k)_{k\ge1}$, 
$$
\left(\frac{\tilde Z_2^\varepsilon(nt)}{\sigma_n}\right)_{t\in[0,1]}\dconv
\left(\Z_2(t)\right)_{t\in[0,1]}
\;\text{ and }\;
\left(\frac{\tilde U_2^\varepsilon(nt)}{\sigma_n}\right)_{t\in[0,1]}\dconv
\left(\U_2(t)\right)_{t\in[0,1]},
$$
in $D([0,1])$, where $\Z_2$ is as in Theorem~\ref{thm:Z} and $\U_2$ is as in Theorem~\ref{thm:U}.
\end{proposition}

\begin{proof}
First remark that, since for all $t\ge 0$, $\tilde q_k(t)=\frac12 \tilde p_k(2t)$, we have $\tilde U_2^\varepsilon(t)=\frac12 \tilde Z_2^\varepsilon(2t)$. Thus the second convergence follows from the first one.
\medskip

\noindent {\em (i) The covariances.}
Since the $\varepsilon_k$ are independent, using \eqref{equiv}, we have for all $t,s\ge 0$
\begin{align*}
\frac{1}{\sigma_n^2}\Cov&(\tilde Z_2^\varepsilon(nt),\tilde Z_2^\varepsilon(ns))
=\frac{1}{\sigma_n^2}\sum_{k\ge 1}\E(\varepsilon_k^2)\tilde p_k(nt)\tilde p_k(ns)
=\frac{1}{\sigma_n^2}\sum_{k\ge 1}(1-e^{-p_k nt})(1-e^{-p_k ns})\\
&=\frac{1}{\sigma_n^2}\left(V(nt)+V(ns)-V(n(t+s))\right)
 \rightarrow \Gamma(1-\alpha)\left( t^\alpha + s^\alpha -(t+s)^\alpha\right) \mbox{ as $n\to\infty$.}
\end{align*}
\medskip

\noindent {\em (ii) Finite-dimensional convergence.}
Since $Z_2^\varepsilon$ is a sum of independent bounded random variables, the finite-dimensional convergence follows from the Cram\'er--Wold device and the Lindeberg central limit theorem.
\medskip

\noindent {\em (iii) Tightness.}
Let $p$ be a positive integer. By Burkh\"older inequality, there exists a constant $C_p>0$ such that for all $0\le s\le t \le 1$,
\begin{align*}
\E\left|\frac{1}{\sigma_n}\left(Z_2^\varepsilon(nt)-Z_2^\varepsilon(ns) \right)\right|^{2p}
&\le C_p\frac{1}{\sigma_n^{2p}}\left( \sum_{k\ge 1} (\tilde p_k(nt)-\tilde p_k(ns))^2\right)^p\\
&\le C_p\frac{1}{\sigma_n^{2p}}\left( \sum_{k\ge 1} \tilde p_k(n(t-s))^2\right)^p
=C_p\left(\frac{V(n(t-s))}{\sigma_n^{2}}\right) ^p.
\end{align*}
We now use Lemma~\ref{lem:V}. Let $\gamma \in(0,\alpha)$. There exists $C_{\gamma }>0$
such that
$$
\E\left|\frac{1}{\sigma_n}\left(Z_2^\varepsilon(nt)-Z_2^\varepsilon(ns) \right)\right|^{2p}
\le C_pC_{\gamma }^p |t-s|^{\gamma  p}\, \text{ uniformly in }|t-s|\in[0,1].
$$
Choosing $p$ such that $\gamma  p>1$, this bound gives the tightness \citep[Theorem 13.5]{Bil99}.
\end{proof}

\subsection{Tightness for $\tilde Z_1^\varepsilon$ and $\tilde U_1^\varepsilon$}\label{sec:tightness}

Recall that $\varepsilon\in\{-1,1\}^\N$ is fixed. Let $G$ be either $\tilde Z^\varepsilon_1$ or $\tilde U^\varepsilon_1$. To show the tightness, we will prove
\equh\label{eq:tightness}
\lim_{\delta\to0}\limsup_{n \to\infty} \P\left(\sup_{|t-s|\le \delta}|G(nt)-G(ns)|\ge \eta \sigma_n\right)=0 \mbox{ for all } \eta>0.
\eque
The tightness then follows from the Corollary of Theorem 13.4 in \citep{Bil99}. 
To prove~\eqref{eq:tightness}, we first show the following two lemmas.
\begin{lemma}\label{lem:momentbound}
 Let $G$ be either $\tilde Z_1^\varepsilon$ or $\tilde U_1^\varepsilon$.
For all integer $p\ge 1$ and $\gamma \in(0,\alpha)$, there exits a constant $C_{p,\gamma }>0$ such that for all $s,t\in[0,1]$, for all $n\ge 1$,
\begin{equation}\label{bound:moment}
\E|G(ns)-G(nt)|^{2p} \le C_{p,\gamma } \left(|t-s|^{\gamma  p}\sigma_n^{2 p}+|t-s|^{\gamma }\sigma_n^2 \right).
\end{equation}
\end{lemma}
\begin{lemma}\label{lem:Poissonbound}
 Let $G$ be either $\tilde Z_1^\varepsilon$ or $\tilde U_1^\varepsilon$.
For all $t\le s \le t+\delta$,
\begin{equation}\label{bound:Poisson}
 |G(t)-G(s)|\le N(t+\delta)-N(t)+\delta, \mbox{ almost surely,}
\end{equation}
where $N$ is the Poisson process in the definition of $\tilde Z_1^\varepsilon$ and $\tilde U_1^\varepsilon$.
\end{lemma}
A chaining argument is then applied to establish the tightness by proving the following. 
\begin{lemma}\label{lem:tight}
 If a process $G$ satisfies \eqref{bound:moment} and \eqref{bound:Poisson} for a Poisson process $N$, then~\eqref{eq:tightness} holds.
\end{lemma}
\begin{proof}[Proof of Lemma~\ref{lem:momentbound}]
We prove for $G=\tilde U_1^\varepsilon$. The case $G=\tilde Z_1^\varepsilon$ can be treated in a similar way and is omitted.
In view of Lemma~\ref{lem:V} it is sufficient to prove that for all $p\ge 1$ and all $0\le s<t\le 1$,
\begin{equation}\label{Vbound}
\E|G(t)-G(s)|^{2p}\le C_p\left(V(2(t-s))^p+ V(2(t-s)) \right),
\end{equation}
with the monotone increasing function $V$ defined in \eqref{equiv}.
We prove it by induction.
For $p=1$, by independence of the $N_k$, we have
\begin{align*}
\E|G(t)-G(s)|^2
&=\sum_{k\ge 1}\Var\left(\ind_{\{N_k(t)\text{ is odd}\}}-\ind_{\{N_k(s)\text{ is odd}\}}\right)\\ 
&\le \sum_{k\ge 1}\E\left(\ind_{\{N_k(t)\text{ is odd}\}}-\ind_{\{N_k(s)\text{ is odd}\}}\right)^2\le \sum_{k\ge 1}\tilde q_k(t-s) =\frac12 V(2(t-s)).
\end{align*}
Let $p\ge 1$ and assume that the property holds for $p-1$. We fix $0<s<t$, and  simplify the notations by setting 
\[
X_k:=\ind_{\{N_k(t)\text{ is odd}\}}-\tilde q_k(t)-\left(\ind_{\{N_k(s)\text{ is odd}\}}-\tilde q_k(s)\right).
\] 
Note that $|X_k|\le 2$ for all $k\ge 1$. Since $(X_k)_{k\ge1}$ is centered and independent, it follows that
\begin{align*}
 \E|G(t)-G(s)|^{2p}
&= \sum_{k_1,\ldots,k_p\ge 1 }\E\left( X_{k_1}^2\cdots X_{k_p}^2\right)\\
&\le \sum_{\substack{k_1,\ldots,k_p\ge 1\\k_1 \notin\{k_2,\ldots,k_p\}}}\E\left(X_{k_1}^2\right)\E\left( X_{k_2}^2\cdots X_{k_p}^2\right)+\sum_{\substack{k_1,\ldots,k_p\ge 1\\k_1 \in\{k_2,\ldots,k_p\}}}\E\left( X_{k_1}^2\cdots X_{k_p}^2\right)\\
&\le \left(\sum_{k_1\ge 1}\E\left(X_{k_1}^2\right)+4(p-1)\right)\sum_{k_2,\ldots,k_p\ge 1}\E\left( X_{k_2}^2\cdots X_{k_p}^2\right).
\end{align*}
By the induction hypothesis, we infer
\begin{align*}
 \E|G(t)-G(s)|^{2p}
&\le \left(\frac12 V(2(t-s))+4(p-1)\right)C_{p-1}\left(V(2(t-s))^{p-1}+ V(2(t-s) \right),
\end{align*}
and we deduce \eqref{Vbound} using the monotonicity of $V$.
\end{proof}

\begin{proof}[Proof of Lemma~\ref{lem:Poissonbound}]
Let $t\le s\le t+\delta$. Recalling~\eqref{subadd:p}, we have
\begin{align*}
|\tilde Z_1^\varepsilon(s)-\tilde Z_1^\varepsilon(t)|
&\le\sum_{k\ge1}\left|\ind_{\{N_k(s)\neq 0\}}-\ind_{\{N_k(t)\neq0\}}\right|+\sum_{k\ge1}|\tilde p_k(s)-\tilde p_k(t)|\\
&\le \sum_{k\ge 1}\ind_{\{N_k(s)-N_k(t)\ne 0\}}+\sum_{k\ge 1} \tilde p_k(s-t)\\
&\le N(s)-N(t)+\E\left(N(s-t)\right)\le N(t+\delta)-N(t)+\delta.
\end{align*}
Similarly, recalling~\eqref{subadd:q},
\begin{align*}
|\tilde U_1^\varepsilon(s)-\tilde U_1^\varepsilon(t)|
&\le\sum_{k\ge1}\left|\ind_{\{N_k(s)\text{ is odd}\}}-\ind_{\{N_k(t)\text{ is odd}\}}\right|+\sum_{k\ge1}|\tilde q_k(s)-\tilde q_k(t)|\\
&\le \sum_{k\ge 1}\ind_{\{N_k(s)-N_k(t)\ne 0\}}+\sum_{k\ge 1} \tilde q_k(s-t)\le N(t+\delta)-N(t)+\delta.
\end{align*}
\end{proof}

\begin{proof}[Proof of Lemma~\ref{lem:tight}]
Let $\eta>0$ be fixed. For $\delta\in(0,1)$ and $r:=\floor{\frac{1}{\delta}}+1$, we set $t_i:=i\delta$ for $i=0,\ldots,r-1$, and $t_r:=1$. By \cite[Theorem 7.4]{Bil99}, we have
\begin{equation}\label{eq:chain0}
\P\left(\sup_{|t-s|\le \delta}|G(nt)-G(ns)|\ge 9\eta \sigma_n\right)\le
\sum_{i=1}^r\P\left(\sup_{t_{i-1}\le s\le t_i}|G(ns)-G(nt_{i-1})|\ge 3\eta \sigma_n\right).
\end{equation}
The sequel of the proof is based on a chaining argument. Fix $i\in\{1,\ldots,r\}$. For all $k\ge 1$, we introduce the subdivision of rank $k$ of the interval $[t_{i-1},t_i]$: 
$$
x_{k,\ell}:=t_{i-1}+\ell\frac{\delta}{2^k}, \; \text{for } k\ge 1\text{ and }\ell=0,\ldots,2^k.
$$
For $s\in[t_{i-1},t_i]$ and $n\ge1$, we define the chain $s_0:=t_{i-1}\le s_1\le\ldots\le s_{k_n}\le s$, where for each $k$, $s_k$ is the largest point among $(x_{k,\ell})_{\ell = 0,\dots,2^k}$ of rank $k$ that is smaller than $s$, and where we choose
\begin{equation}\label{def:kn}
k_n:=\floor{\log_2\left(2(e-1)\frac{n\delta}{\eta\sigma_n}\right)}+1.
\end{equation}
This choice of $k_n$ will become clearer later. 
For $t_{i-1}\le s\le t_i$, we write
\begin{equation}\label{eq:chain1}
|G(ns)-G(nt_{i-1})|
\le \sum_{k=1}^{k_n}|G(ns_k)-G(ns_{k-1})| + |G(ns)-G(ns_{k_n})|.
\end{equation}
Since we necessarily have $s_k=s_{k-1}$ or $s_k=s_{k-1}+\frac{\delta}{2^k}$, we infer that for all $k\ge 1$,
\begin{equation}\label{eq:chain2}
|G(ns_k)-G(ns_{k-1})|
\le \max_{\ell=1,\ldots, 2^k}|G(nx_{k,\ell})-G(n x_{k,\ell-1})|.
\end{equation}
Now, by Lemme~\ref{lem:Poissonbound}, we get 
\begin{align}
|G(ns)-G(ns_{k_n})|
&\le N(n(s_{k_n}+\delta2^{-k_n}))-N(ns_{k_n}) + n\delta2^{-k_n} \nonumber\\
&\le \max_{\ell=0,\ldots, 2^{k_n}-1} \pp{N(n(x_{k_n,\ell}+ \delta2^{-k_n}))-N(nx_{k_n,\ell})} + n\delta2^{-k_n}.\label{eq:chain3}
\end{align}
Observe that our choice of $k_n$ in \eqref{def:kn} gives $n\delta2^{-k_n}\le \eta \sigma_n$. By \eqref{eq:chain1}, \eqref{eq:chain2} and \eqref{eq:chain3}, we infer
\begin{align}
&\hspace{-30pt}\limsup_{n\to\infty} \P\left(\sup_{t_{i-1}\le s\le t_i}|G(ns)-G(nt_{i-1})|\ge 3\eta \sigma_n\right)\nonumber\\
&\le \limsup_{n\to\infty} \P\left(\sum_{k=1}^{k_n}\max_{\ell=1,\ldots, 2^k}|G(nx_{k,\ell})-G(n x_{k,\ell-1})|>\eta \sigma_n\right)\label{eq:chain4}\\
& \quad+ \limsup_{n\to\infty} \P\left(\max_{\ell=0,\ldots, 2^{k_n}-1} \pp{N(n(x_{k_n,\ell}+ \delta2^{-k_n}))-N(nx_{k_n,\ell})}>\eta \sigma_n\right).\label{eq:chain5}
\end{align}
For~\eqref{eq:chain5}, using exponential Markov inequality and the fact that $\E(e^{N(x)})=e^{x(e-1)}$,  we infer
\begin{multline*}
\P\left(\max_{\ell=1,\ldots, 2^k} \left\{N(n(x_{k_n,\ell}+ \delta2^{-k_n}))-N(nx_{k_n,\ell})\right\}>\eta \sigma_n\right)\\
\le 2^{k_n}\P\left(N(n\delta2^{-k_n})>\eta \sigma_n\right)\leq 2^{k_n} e^{n\delta 2^{-k_n}(e-1)-\eta\sigma_n}.
\end{multline*}
Again by the choice of $k_n$ in \eqref{def:kn}, $2^{k_n}\leq 4(e-1)n\delta/(\eta\sigma_n)$ and $2^{-k_n}\leq \eta\sigma_n/(2(e-1)n\delta)$. Thus, the above inequality is bounded by 
$4(e-1){n\delta}/(\eta\sigma_n)e^{-\frac12\eta\sigma_n}$, which converges to 0 as $n\to\infty$.
So, the term \eqref{eq:chain5} vanishes and it remains to deal with \eqref{eq:chain4}.
Let $\eta_k:=\frac{\eta}{k(k+1)}$, $k\ge 1$, so that $\sum_{k\ge 1}\eta_k=\eta$. We have
\begin{align*}
&\hspace{-50pt}\P\left(\sum_{k=1}^{k_n}\max_{\ell=1,\ldots, 2^k}|G(nx_{k,\ell})-G(n x_{k,\ell-1})|>\eta \sigma_n\right)\\
&\le \sum_{k=1}^{k_n}\P\left(\max_{\ell=1,\ldots, 2^k}|G(nx_{k,\ell})-G(n x_{k,\ell-1})|>\eta_k \sigma_n\right)\\
&\le \sum_{k=1}^{k_n}\sum_{\ell=1}^{2^k}\P\left(|G(nx_{k,\ell})-G(n x_{k,\ell-1})|>\eta_k \sigma_n\right).
\end{align*}
Now, fix $\gamma \in(0,\alpha)$ and let $p\ge 1$ be an integer such that $\gamma  p>1$. Using Markov inequality at order $2p$ and the $2p$-th moment bound \eqref{bound:moment}, we get
 \begin{align*}
\P\Bigg(\sum_{k=1}^{k_n}&\max_{\ell=1,\ldots, 2^k}|G(nx_{k,\ell})-G(n x_{k,\ell-1})|>\eta \sigma_n\Bigg)
\le \sum_{k=1}^{k_n}\sum_{\ell=1}^{2^k}\eta_k^{-2p} \frac{\E\left|G(nx_{k,\ell})-G(n x_{k,\ell-1})  \right|^{2p}}{\sigma_n^{2p}}\\
&\le C_{p,\gamma } \sum_{k=1}^{k_n}\sum_{\ell=1}^{2^k}\eta_k^{-2p} \left(|x_{k,\ell}- x_{k,\ell-1}|^{\gamma  p}+ \frac{|x_{k,\ell}- x_{k,\ell-1}|^{\gamma } }{\sigma_n^{2(p-1)}}\right)\\
&\le C_{p,\gamma } \delta^{\gamma  p}\sum_{k=1}^{\infty}\eta_k^{-2p}2^{k(1-\gamma  p)}
+C_{p,\gamma }\delta^{\gamma } n^{\alpha(1-p)}L(n)^{1-p} \sum_{k=1}^{k_n}\eta_k^{-2p}2^{k(1-\gamma )}.
\end{align*}
In the right-hand side, since $\gamma p>1$, the series in the first term is converging and is independent of $n$. The sum in the second term is bounded, up to a multiplicative constant, by $2^{k_n(1-\gamma )}$ which is of order $n^{(1-\alpha/2)(1-\gamma )}$ (here and next line, up to a slowly varying function). Thus, the second term in the right-hand side is of order 
$n^{1-\alpha p+\alpha/2-\gamma + \gamma \alpha/2}\le n^{1-\gamma  p +(\alpha-\gamma )(1-p)}$ and vanishes as $n$ goes to $\infty$, again because we have assumed $\gamma  p>1$. So for~\eqref{eq:chain4}, we arrive at
\[
\limsup_{n\to\infty} \P\left(\sum_{k=1}^{k_n}\max_{\ell=1,\ldots, 2^k}|G(nx_{k,\ell})-G(n x_{k,\ell-1})|>\eta \sigma_n\right) \leq C\delta^{\gamma p}
\]
for some constant $C$ independent of $\delta$ and $\eta$. 
From \eqref{eq:chain0}, we conclude that \begin{align*}
\limsup_{n\to\infty} \P\left(\sup_{|t-s|\le \delta}|G(nt)-G(ns)|\ge 9\eta \sigma_n\right)
\le C' \left(\floor{\frac{1}{\delta}}+1\right)\delta^{\gamma  p}
\end{align*}
which goes to $0$ as $\delta\downarrow 0$. This yields~\eqref{eq:tightness}.
\end{proof}
\begin{remark}
For the Poissonized model, we can establish similar weak convergence to the decompositions as in Theroems~\ref{thm:Z} and~\ref{thm:U}, by adapting the proofs at the end of Section~\ref{sec:dP}. We omit this part. 
\end{remark}

\section{De-Poissonization}\label{sec:dP}

In this section we prove our main theorems. Recall the decompositions
\[
Z^\varepsilon = Z_1^\varepsilon+Z_2^\varepsilon \qmand U^\varepsilon = U_1^\varepsilon+U_2^\varepsilon,
\]
and
\[
\tilde Z^\varepsilon = \tilde Z_1^\varepsilon+\tilde Z_2^\varepsilon \qmand \tilde U^\varepsilon = \tilde U_1^\varepsilon+\tilde U_2^\varepsilon.
\]
Note that $G^\varepsilon$ and $\tilde G^\varepsilon$, for $G$ being $Z_1,Z_2,U_1,U_2$ respectively, are coupled in the sense that they are defined on the same probability space as functionals of the same $\varepsilon$ and $(Y_n)_{n\ge1}$. We have already established weak convergence results for $\tilde Z_1^\varepsilon,\tilde Z_2^\varepsilon, \tilde U_1^\varepsilon, \tilde U_2^\varepsilon$. The de-Poissonization step thus consists of controlling the distance between $G^\varepsilon$ and $\tilde G^\varepsilon$. 
We first prove the easier part.

\subsection{The processes $Z_2^\varepsilon$ and $U_2^\varepsilon$}
\begin{theorem}\label{thm:2}
For a Rademacher sequence $\varepsilon$,
$$
\left(\frac{Z_2^\varepsilon(\floor{nt})}{\sigma_n}\right)_{t\in[0,1]}\dconv
\left(\Z_2(t)\right)_{t\in[0,1]}\qmand \left(\frac{U_2^\varepsilon(\floor{nt})}{\sigma_n}\right)_{t\in[0,1]}\dconv
\left(\U_2(t)\right)_{t\in[0,1]},
$$
in $D([0,1])$, where $\Z_2$ and $\U_2$ are as in Theorems~\ref{thm:Z} and~\ref{thm:U}.
\end{theorem}

\begin{proof}%[Proof of Theorem~\ref{thm:Z_2} and Theorem~\ref{thm:U_2}]
Thanks to the coupling, it suffices to show
for all $\varepsilon\in\{-1,1\}^\N$ fixed, 
\[
\lim_{n\to\infty}\sup_{t\in[0,1]}\frac{|\tilde G^\varepsilon(\floor{nt})- G^\varepsilon(nt)|}{\sigma_n} = 0 
\]
in probability,
with $G$ being $Z_2, U_2$ respectively. We actually prove the above convergence in the almost sure sense. Observe that for all $\varepsilon\in\{-1,1\}^\N$, 
\begin{align*}
|\tilde Z_2^\varepsilon(\floor{nt})- Z_2^\varepsilon(nt)| &\leq \sum_{k\geq 1}|\tilde  p_k(nt) - p_k(\floor{nt})|,\\
|\tilde U_2^\varepsilon(\floor{nt})-U_2^\varepsilon(nt)| &\leq \sum_{k\geq 1}|\tilde q_k(nt) - q_k(\floor{nt}) |.
\end{align*}
Thus, the proof is completed once the following 
 Lemma~\ref{lem:approx} is proved. 
\end{proof}

\begin{lemma}\label{lem:approx}
The following limits hold:
\begin{equation}\label{approx:p_k}
\lim_{n\to\infty}\frac{1}{\sigma_n}\sup_{t\in[0,1]}\sum_{k\ge 1}|\tilde p_k(nt)-p_k(\floor{nt})| =0
\end{equation}
and
\begin{equation}\label{approx:q_k}
\lim_{n\to\infty}\frac{1}{\sigma_n}\sup_{t\in[0,1]}\sum_{k\ge 1}|\tilde q_k(nt)-q_k(\floor{nt})| =0.
\end{equation}
\end{lemma}
\begin{proof}
By triangular inequality, for all $n\ge 1$, $t\ge0$,
$$
\sum_{k\ge 1}|\tilde p_k(nt)-p_k(\floor{nt})|\le\sum_{k\ge 1} |\tilde p_k(\floor{nt})- \tilde p_k(nt)|+\sum_{k\ge 1}|\tilde p_k(\floor{nt})-p_k(\floor{nt})|.
$$
First, note that for all $k\ge 1$, 
$$
|\tilde p_k(\floor{nt})- \tilde p_k(nt)|\le \tilde p_k(\floor{nt}+1)- \tilde p_k(\floor{nt})=e^{-p_k\floor{nt}}(1-e^{-p_k}),
$$
and thus,
$$
\sum_{k\ge 1}|\tilde p_k(\floor{nt})- \tilde p_k(nt)|\le \sum_{k\ge 1}p_k=1.
$$
Further, if $\floor{nt}\ge 1$, using that $e^{-my}-(1-y)^m\le \frac1m(1-e^{-my})$ for all $0\le y\le 1$ and $m\in\N$, we have
\begin{align*}
\sum_{k\ge 1}|\tilde p_k(\floor{nt})-p_k(\floor{nt})|
&= \sum_{k\ge 1} \left(e^{-p_k\floor{nt}}-(1-p_k)^{\floor{nt}}\right)\\
&\le \frac{1}{\floor{nt}} \sum_{k\ge 1}(1-e^{-p_k\floor{nt}})
= \frac{V(\floor{nt})}{\floor{nt}},
\end{align*}
which is bounded (since $V(n)/n\to0$ as $n\to\infty$). We thus deduce \eqref{approx:p_k}. The proof for \eqref{approx:q_k} is similar and omitted.
\end{proof}

\subsection{The processes $Z_1^\varepsilon$ and $U_1^\varepsilon$}
In this section we prove Theorem~\ref{thm:1}.
The coupling of $Z_1^\varepsilon,\tilde Z_1^\varepsilon$ and $U_1^\varepsilon,\tilde U_1^\varepsilon$ respectively  takes a little more effort to control.

 \begin{proof}[Proof of Theorem~\ref{thm:1}]
Let $N$ be the Poisson process introduced in Section~\ref{sec:Poi} and denote by $\tau_i$ the $i$-th arrival time of $N$, $i\ge 1$, namely $\tau_i:=\inf\{t>0\mid N(t)=i\}$.
We introduce the random changes of time $\lambda_n:[0,\infty)\to [0,\infty)$, $n\ge 1$, given by
$$
\lambda_n(t):=\frac{\tau_{\floor{nt}}}{n},\quad t\ge 0.
$$
By constructions, we have
$$
Z^\varepsilon(\floor{nt})=\tilde Z^\varepsilon(n\lambda_n(t))
\qmand
\tilde U^\varepsilon(\floor{nt})=U^\varepsilon(n\lambda_n(t)), \mbox{ amost surely.}
$$
These identities do not hold for the process $Z_1^\varepsilon$ or $U_1^\varepsilon$ but we can still couple $Z_1^\varepsilon, \tilde Z_1^\varepsilon$ and $U_1^\varepsilon, \tilde U_1^\varepsilon$ via
\begin{align}\label{decomp:Z}
Z_1^\varepsilon(\floor{nt})&=\tilde Z_1^\varepsilon(n\lambda_n(t))+\sum_{k\ge 1}\varepsilon_k(\tilde p_k(n\lambda_n(t))-p_k(\floor{nt}))
\\
\label{decomp:U}
U_1^\varepsilon(\floor{nt})&=\tilde U_1^\varepsilon(n\lambda_n(t))+\sum_{k\ge 1}\varepsilon_k(\tilde q_k(n\lambda_n(t))-q_k(\floor{nt})).
\end{align}
The proof is now decomposed into two lemmas treating separately the two terms in the right-hand side of the preceding identities.
\begin{lemma}\label{lem:conv_lambda}
We have
$$
\left(\frac{\tilde Z_1^\varepsilon(n\lambda_n(t))}{\sigma_n}\right)_{t\in[0,1]}\dconv(\Z_1(t))_{t\in[0,1]}
\qmand
\left(\frac{\tilde U_1^\varepsilon(n\lambda_n(t))}{\sigma_n}\right)_{t\in[0,1]}\dconv(\U_1(t))_{t\in[0,1]}
$$
in $D([0,1])$.
\end{lemma}
\begin{proof}
We only prove the first convergence. The proof of the second is the same by replacing $(\tilde Z_1^\varepsilon,\Z_1)$ by $(\tilde U_1^\varepsilon,\U_1)$. 
For $t\ge0$, by the law of large numbers, $\lambda_n(t)\to t$ almost surely as $n\to\infty$. Since the $\lambda_n$ are nondecreasing, almost surely the convergence holds for all $t\ge0$, and by P\'olya's extension of Dini's theorem (see \cite[Problem 127]{PolSze72}) the convergence is uniform for $t$ in a compact interval.
That is
\[
\lim_{n\to\infty}\sup_{t\in[0,1]}|\lambda_n(t)-t|=0 \mbox{ almost surely},
\]
and $\lambda_n$ converges almost surely to the identity function $\mathbb I$ in $D([0,1])$. 

We want to apply the random change of time lemma from \citet[p.~151]{Bil99}. However, $\lambda_n$ is not a good candidate as it is not bounded between $[0,1]$. Instead, we introduce
\[
\lambda_n^*(t) :=\min\left(\lambda_n(t), 1\right),\quad t\geq 0.
\]
Observe that by monotonicity, 
\[
\sup_{t\in[0,1]}|\lambda_n^*(t)-t|\leq \sup_{t\in[0,1]}|\lambda_n(t)-t|.
\]
Thus, $\lambda_n^*$ converges almost surely to $\mathbb I$ in $D([0,1])$. 
By Slutsky's lemma and Proposition \ref{prop:Poisson1}, we also have 
\begin{equation}\label{eq:lambda:Z}
\left(\left(\frac{\tilde Z_1^\varepsilon(nt)}{\sigma_n}\right)_{t\in[0,1]},(\lambda^*_n(t))_{t\in[0,1]}\right)\dconv \left((\Z_1(t))_{t\in[0,1]}, {\mathbb I} \right)
\end{equation}
in $D([0,1])\times D([0,1])$. 
Furthermore, since $\lambda_n^*$ is non-decreasing and bounded in $[0,1]$, thus by random change of time lemma we obtain 
\equh\label{eq:lambdan*}
\left(\frac{\tilde Z_1^\varepsilon(n\lambda_n^*(t))}{\sigma_n}\right)_{t\in[0,1]}\dconv(\Z_1(t))_{t\in[0,1]}
\eque
in $D([0,1])$. To obtain the desired result we need to replace $\lambda_n^*$ by $\lambda_n$. However, by definition, we only have, for all $\eta\in(0,1)$ fixed, 
\[
\P(\lambda_n^*\ne\lambda_n\text{ on }[0,1-\eta])\le\P\pp{\tau_{\floor{n(1-\eta)}}\ge  n}\to0\;\text{ as }n\to\infty.
\]
It then follows that, restricting the convergence of~\eqref{eq:lambdan*} in $D([0,1-\eta])$, 
\[
\pp{\frac{\tilde Z_1^\varepsilon(n\lambda_n(t))}{\sigma_n}}_{t\in[0,1-\eta]}\Rightarrow (\Z_1(t))_{t\in[0,1-\eta]}
\]
in $D([0,1-\eta])$.
This is strictly weaker than the convergence in $D([0,1])$ that we are looking for. However, looking back we see an easy fix as follows.
If one starts in~\eqref{eq:lambda:Z} with weak convergence for $\tilde Z_1^\varepsilon$ and $\lambda_n^*$ (modified accordingly) as processes indexed by  a slightly larger time interval, say in $D([0,1/(1-\eta)])$ for any $\eta\in(0,1)$ fixed, the desired result then follows.
\end{proof}

In view of Lemma~\ref{lem:approx}, the following lemma will be sufficient to conclude.
\begin{lemma}\label{lem:approx_lambda}
 The following limits hold:
\begin{equation*}%\label{approx:lambda-p_k}
\lim_{n\to\infty}\frac{1}{\sigma_n}\sup_{t\in[0,1]}\sum_{k\ge 1}|\tilde p_k(nt)-\tilde p_k(n\lambda_n(t))| =0\text{ in probability}
\end{equation*}
and
\begin{equation}\label{approx:lambda-q_k}
\lim_{n\to\infty}\frac{1}{\sigma_n}\sup_{t\in[0,1]}\sum_{k\ge 1}|\tilde q_k(nt)-\tilde q_k(n\lambda_n(t))| =0\text{ in probability}.
\end{equation}
\end{lemma}
\begin{proof}
We only prove the second limit. The first one can be proved in a similar way and is omitted. 
We first introduce 
\[
\Lambda_n(t):=n^{\frac12}(\lambda_n(t)-t)={n^{-\frac12}}(\tau_{\floor{nt}}-nt).
\]
Since $\tau_n$ is the sum of i.i.d.\ random variables with exponential distribution of rate $1$, and since $n^{-\frac12}(nt-\floor{nt})$ converges to $0$ uniformly in $t$, by Donsker's theorem and Slutsky's lemma, we have
$$
\left(\Lambda_n(t)\right)_{t\in[0,1]}\dconv \left(\B(t)\right)_{t\in[0,1]}\text{ in }D([0,1]),
$$
where $\B$ is a standard Brownian motion. By the continuous mapping theorem, the sequence $\sup_{t\in[0,1]} |\Lambda_n(t)|$ weakly converges to $\sup_{t\in[0,1]}|\B(t)|$, as $n\to\infty$. In particular, $(\sup_{t\in[0,1]} |\Lambda_n(t)|)_{n\ge1}$ is tight. So, for any $\eta>0$, there exits $K_\eta>0$ such that for $n$ large enough,
\begin{equation}\label{Lambda:tight}
\P\left(\sup_{t\in[0,1]} |\Lambda_n(t)|>K_\eta\right)\le\eta.
\end{equation}

\medskip

Now, choose $\beta\in(0,1/2)$ and consider
$$
A_n:=\sup_{t\in[0,n^{-\beta}]}\sum_{k\ge 1}|\tilde q_k(nt)-\tilde q_k(n\lambda_n(t))|\;\text{ and }\; B_n:=\sup_{t\in[n^{-\beta},1]}\sum_{k\ge 1}|\tilde q_k(nt)-\tilde q_k(n\lambda_n(t))|.
$$
Concerning $A_n$, using the bound in \eqref{subadd:q}, we have
$$
A_n\le \sup_{t\in[0,n^{-\beta}]}\sum_{k\ge 1}\tilde q_k(n|\lambda_n(t)-t|)\\
=\sup_{t\in[0,1]}\sum_{k\ge 1}\tilde q_k(n|\lambda_n(n^{-\beta}t)-n^{-\beta}t|).
$$
We can write
\[
\lambda_n(n^{-\beta}t)-n^{-\beta}t=\frac{\Lambda_{n^{1-\beta}}(t)}{n^{\frac{1+\beta}{2}}}.
\]
For any $\eta>0$, using \eqref{Lambda:tight}, by monotonicity of $\tilde q_k(\cdot)$, we infer that  for $n$ large enough 
$$
\P\left(A_n\le \sum_{k\ge 1}\tilde q_k\left(n\cdot n^{-\frac{1+\beta}{2}}K_\eta \right)\right)>1-\eta.
$$
But
\begin{align*}
\frac{1}{\sigma_n}\sum_{k\ge 1}\tilde q_k\left(n^{1-(1+\beta)/2}K_\eta \right)
&=\frac{1}{2\sigma_n}V\left(2n^{(1-\beta)/2}K_\eta \right)\\
&\sim \Gamma(1-\alpha)2^{\alpha-1} K_\eta^\alpha n^{-\beta\alpha/2}\frac{L(n^{(1-\beta)/2})}{L(n)^{1/2}} \to 0 \mbox{ as $n\to\infty$}.
\end{align*}
Thus, $A_n/\sigma_n$ converges to $0$ in probability as $n$ goes to $\infty$.

Concerning $B_n$, using the identity \eqref{subadd:q}, we can write
\begin{align*}
B_n&=\sup_{t\in[n^{-\beta},1]}\sum_{k\ge 1}\left(1-2\tilde q_k(n\min(\lambda_n(t), t))\right)\tilde q_k(n|\lambda_n(t)-t|)\\
&=\sup_{t\in[n^{-\beta},1]}\sum_{k\ge 1}e^{-2p_kn\min(\lambda_n(t),t)}\tilde q_k(n|\lambda_n(t)-t|).
\end{align*}
Now, for $t\in[n^{-\beta},1]$, observe that if for some $K>0$, $|\Lambda_n(t)|\leq K$  and $n^{\frac12-\beta}>2K$, then
$$
\lambda_n(t) =t+ n^{-1/2}\Lambda_n(t) \ge t-n^{-1/2}|\Lambda_n(t)|\ge t-\frac{n^{-\beta}}{2}\ge \frac{t}{2},
$$
and thus $\min(\lambda_n(t),t) \ge \frac{t}{2}$.
Let $\eta>0$ and $K_\eta$ be as in \eqref{Lambda:tight}. Assume $n$ is large enough so that \eqref{Lambda:tight} holds and $n^{\frac12-\beta}>2K_\eta$ (which is possible since we have chosen $\beta\in(0,1/2)$). By the preceding observation and by monotonicity of $\tilde q_k(\cdot)$, we infer
$$
\P\left(B_n\le \sup_{t\in[n^{-\beta},1]}\sum_{k\ge 1}e^{-p_knt}\tilde q_k\left(n\cdot n^{-\frac12} K_\eta\right)\right)>1-\eta.
$$
Now, using $1-e^{-x}\le x$ and then $xe^{-x}\le 1-e^{-x}$ for $x>0$, we get 
\begin{multline*}
\sup_{t\in[n^{-\beta},1]}\sum_{k\ge 1}e^{-p_knt}\tilde q_k\left(n^{1-\frac12} K_\eta\right)
= \sum_{k\ge 1}e^{-p_kn^{1-\beta}}\frac12\left(1-e^{-2p_k n^{\frac12} K_\eta}\right)\\
\leq \sum_{k\ge 1}e^{-p_kn^{1-\beta}}p_kn^{\frac12} K_\eta
 \le  \sum_{k\ge 1}\left(1-e^{-p_kn^{1-\beta}}\right)n^{-\frac12+\beta} K_\eta= {n^{\beta-1/2}} V(n^{1-\beta})K_\eta.
\end{multline*}
Thus,
\begin{multline*}
\frac1{\sigma_n}\sup_{t\in[n^{-\beta},1]}\sum_{k\ge 1}e^{-p_knt}\tilde q_k\left(n^{1-\frac12} K_\eta\right) 
 \leq \frac{n^{\beta-1/2}V(n^{1-\beta})}{\sigma_{n}}K_\eta
 \\
\sim \Gamma(1-\alpha)K_\eta n^{(\beta-\frac12)(1-\alpha)}\frac{L(n^{(1-\beta)})}{L(n)^{1/2}}\to 0 \mbox{ as $n\to\infty$},
\end{multline*}
since $\beta\in(0,1/2)$. Thus $B_n/\sigma_n$ converges to $0$ in probability as $n$ goes to $\infty$. We have thus proved~\eqref{approx:lambda-q_k}.
\end{proof}
To sum up, the desired results now follow from \eqref{decomp:Z} and \eqref{decomp:U}, Lemmas \ref{lem:approx}, \ref{lem:conv_lambda} and \ref{lem:approx_lambda}, and Slutsky's lemma.
\end{proof}

\subsection{The trivariate processes}
Finally we conclude by establishing the main theorems.
\begin{proof}[Proof of Theorems~\ref{thm:Z} and~\ref{thm:U}]
We prove Theorem~\ref{thm:Z}. The proof for Theorem~\ref{thm:U} is the same. We denote by $\cE$ the $\sigma$-field generated by the $(\varepsilon_k)_{k\ge 1}$ which is then independent of $(Y_n)_{n\ge 1}$. Note that the process $Z_2^\varepsilon$ is $\cE$-measurable.
For any continuous and bounded function $f$ and $g$ from $D([0,1])$ to $\R$, we have 
\begin{align*}
\bigg|\E&\bigg(f\left(\frac{Z_1^\varepsilon(\floor{n\cdot})}{\sigma_n}\right)  g\left(\frac{Z_2^\varepsilon(\floor{n\cdot})}{\sigma_n}\right)\bigg)
-\E f(\Z_1)\E g(\Z_2)\bigg|\\
&=\left|\E\left[\E\left(f\left(\frac{Z_1^\varepsilon(\floor{n\cdot})}{\sigma_n}\right)\mmid \cE\right)g\left(\frac{Z_2^\varepsilon(\floor{n\cdot})}{\sigma_n}\right)\right]
-\E f(\Z_1)\E g(\Z_2)\right|\\
&\le \E\left|\E\left(f\left(\frac{Z_1^\varepsilon(\floor{n\cdot})}{\sigma_n}\right)\mmid \cE\right)-\E f(\Z_1)\right|\cdot \|g\|_\infty+\left|\E g\left(\frac{Z_2^\varepsilon(\floor{n\cdot})}{\sigma_n}\right)-\E g(\Z_2)\right|\cdot \|f\|_\infty.
\end{align*}
The first term goes to $0$ as $n\to\infty$ thanks to Theorem~\ref{thm:1} and the dominated convergence theorem. The second one goes to $0$ as $n\to\infty$ thanks to Theorem~\ref{thm:2}. By \cite[Corollary 1.4.5]{VaaWel96} we deduce that 
$$
\frac{1}{\sigma_n}\left(Z_1^\varepsilon(\floor{nt}),Z_2^\varepsilon(\floor{nt})\right)_{t\in[0,1]}\dconv \left(\Z_1(t),\Z_2(t)\right)_{t\in[0,1]},
$$
in $D([0,1])^2$ where $\Z_1$ and $\Z_2$ are independent. The rest of the theorem follows from the identity $Z^\varepsilon=Z_1^\varepsilon+Z_2^\varepsilon$.
\end{proof}

\subsection*{Acknowledgments} 
The authors would like to thank David Nualart and Gennady Samorodnitsky for helpful discussions.
The first author would like to thank the hospitality and financial support from Taft Research Center and Department of Mathematical Sciences at University of Cincinnati, for his visit in May and June 2015, when most of the results were obtained. The first author's research was partially supported by the R\'egion Centre project MADACA. The second author would like to thank the invitation and hospitality of 
Laboratoire de Math\'ematiques et Physique Th\'eorique, UMR-CNRS 7350, Tours, France, for his visit from April to July in 2014, when the project was initiated. The second author's research was partially supported by NSA grant H98230-14-1-0318.

\bibliographystyle{apalike}
\bibliography{references,references1}

\def\cprime{$'$} \def\polhk#1{\setbox0=\hbox{#1}{\ooalign{\hidewidth
  \lower1.5ex\hbox{`}\hidewidth\crcr\unhbox0}}}
  \def\polhk#1{\setbox0=\hbox{#1}{\ooalign{\hidewidth
  \lower1.5ex\hbox{`}\hidewidth\crcr\unhbox0}}}
\begin{thebibliography}{}

\bibitem[Bahadur, 1960]{Bah60}
Bahadur, R.~R. (1960).
\newblock On the number of distinct values in a large sample from an infinite
  discrete distribution.
\newblock {\em Proc. Nat. Inst. Sci. India Part A}, 26(supplement II):67--75.

\bibitem[Billingsley, 1999]{Bil99}
Billingsley, P. (1999).
\newblock {\em Convergence of probability measures}.
\newblock Wiley Series in Probability and Statistics: Probability and
  Statistics. John Wiley \& Sons, Inc., New York, second edition.
\newblock A Wiley-Interscience Publication.

\bibitem[Bingham et~al., 1987]{BinGolTeu87}
Bingham, N.~H., Goldie, C.~M., and Teugels, J.~L. (1987).
\newblock {\em Regular variation}, volume~27 of {\em Encyclopedia of
  Mathematics and its Applications}.
\newblock Cambridge University Press, Cambridge.

\bibitem[Bojdecki et~al., 2004]{BojGorTal04}
Bojdecki, T., Gorostiza, L.~G., and Talarczyk, A. (2004).
\newblock Sub-fractional {B}rownian motion and its relation to occupation
  times.
\newblock {\em Statist. Probab. Lett.}, 69(4):405--419.

\bibitem[Bojdecki and Talarczyk, 2012]{bojdecki12particle}
Bojdecki, T. and Talarczyk, A. (2012).
\newblock Particle picture interpretation of some {G}aussian processes related
  to fractional {B}rownian motion.
\newblock {\em Stochastic Process. Appl.}, 122(5):2134--2154.

\bibitem[Bunge and Fitzpatrick, 1993]{BunFit93}
Bunge, J. and Fitzpatrick, M. (1993).
\newblock Estimating the number of species: a review.
\newblock {\em J. Am. Stat. Ass.}, 88(421):364--373.

\bibitem[Davydov, 1970]{davydov70invariance}
Davydov, J.~A. (1970).
\newblock The invariance principle for stationary processes.
\newblock {\em Teor. Verojatnost. i Primenen.}, 15:498--509.

\bibitem[Dzhaparidze and van Zanten, 2004]{DzhZan04}
Dzhaparidze, K. and van Zanten, H. (2004).
\newblock A series expansion of fractional {B}rownian motion.
\newblock {\em Probab. Theory Related Fields}, 130(1):39--55.

\bibitem[Embrechts and Maejima, 2002]{embrechts02selfsimilar}
Embrechts, P. and Maejima, M. (2002).
\newblock {\em Selfsimilar processes}.
\newblock Princeton Series in Applied Mathematics. Princeton University Press,
  Princeton, NJ.

\bibitem[Enriquez, 2004]{Enr04}
Enriquez, N. (2004).
\newblock A simple construction of the fractional {B}rownian motion.
\newblock {\em Stochastic Process. Appl.}, 109(2):203--223.

\bibitem[Gnedin et~al., 2007]{GneHanPit07}
Gnedin, A., Hansen, B., and Pitman, J. (2007).
\newblock Notes on the occupancy problem with infinitely many boxes: general
  asymptotics and power laws.
\newblock {\em Probab. Surv.}, 4:146--171.

\bibitem[Hammond and Sheffield, 2013]{HamShe13}
Hammond, A. and Sheffield, S. (2013).
\newblock Power law {P}\'olya's urn and fractional {B}rownian motion.
\newblock {\em Probab. Theory Related Fields}, 157(3-4):691--719.

\bibitem[Houdr{\'e} and Villa, 2003]{houdre03example}
Houdr{\'e}, C. and Villa, J. (2003).
\newblock An example of infinite dimensional quasi-helix.
\newblock In {\em Stochastic models ({M}exico {C}ity, 2002)}, volume 336 of
  {\em Contemp. Math.}, pages 195--201. Amer. Math. Soc., Providence, RI.

\bibitem[Karlin, 1967]{Kar67}
Karlin, S. (1967).
\newblock Central limit theorems for certain infinite urn schemes.
\newblock {\em J. Math. Mech.}, 17:373--401.

\bibitem[Kl{\"u}ppelberg and K{\"u}hn, 2004]{kluppelberg04fractional}
Kl{\"u}ppelberg, C. and K{\"u}hn, C. (2004).
\newblock Fractional {B}rownian motion as a weak limit of {P}oisson shot noise
  processes---with applications to finance.
\newblock {\em Stochastic Process. Appl.}, 113(2):333--351.

\bibitem[Kolmogorov, 1940]{kolmogorov40wienersche}
Kolmogorov, A.~N. (1940).
\newblock Wienersche {S}piralen und einige andere interessante {K}urven im
  {H}ilbertschen {R}aum.
\newblock {\em C. R. (Doklady) Acad. Sci. URSS (N.S.)}, 26:115--118.

\bibitem[Lei and Nualart, 2009]{lei09decomposition}
Lei, P. and Nualart, D. (2009).
\newblock A decomposition of the bifractional {B}rownian motion and some
  applications.
\newblock {\em Statist. Probab. Lett.}, 79(5):619--624.

\bibitem[Mandelbrot and Van~Ness, 1968]{mandelbrot68fractional}
Mandelbrot, B.~B. and Van~Ness, J.~W. (1968).
\newblock Fractional {B}rownian motions, fractional noises and applications.
\newblock {\em SIAM Rev.}, 10:422--437.

\bibitem[Mikosch and Samorodnitsky, 2007]{mikosch07scaling}
Mikosch, T. and Samorodnitsky, G. (2007).
\newblock Scaling limits for cumulative input processes.
\newblock {\em Math. Oper. Res.}, 32(4):890--918.

\bibitem[Peligrad and Sethuraman, 2008]{peligrad08fractional}
Peligrad, M. and Sethuraman, S. (2008).
\newblock On fractional {B}rownian motion limits in one dimensional
  nearest-neighbor symmetric simple exclusion.
\newblock {\em ALEA Lat. Am. J. Probab. Math. Stat.}, 4:245--255.

\bibitem[Pipiras and Taqqu, 2015]{pipiras15long}
Pipiras, V. and Taqqu, M. (2015).
\newblock Long-range dependence and self-similarity.
\newblock Cambridge University Press, Forthcoming in 2015.

\bibitem[Pitman, 2006]{pitman06combinatorial}
Pitman, J. (2006).
\newblock {\em Combinatorial stochastic processes}, volume 1875 of {\em Lecture
  Notes in Mathematics}.
\newblock Springer-Verlag, Berlin.
\newblock Lectures from the 32nd Summer School on Probability Theory held in
  Saint-Flour, July 7--24, 2002, With a foreword by Jean Picard.

\bibitem[P{\'o}lya and Szeg{\H{o}}, 1972]{PolSze72}
P{\'o}lya, G. and Szeg{\H{o}}, G. (1972).
\newblock {\em Problems and theorems in analysis. {V}ol. {I}: {S}eries,
  integral calculus, theory of functions}.
\newblock Springer-Verlag, New York-Berlin.

\bibitem[Ruiz~de Ch{\'a}vez and Tudor, 2009]{ChaTud09}
Ruiz~de Ch{\'a}vez, J. and Tudor, C. (2009).
\newblock A decomposition of sub-fractional {B}rownian motion.
\newblock {\em Math. Rep. (Bucur.)}, 11(61)(1):67--74.

\bibitem[Samorodnitsky, 2006]{samorodnitsky06long}
Samorodnitsky, G. (2006).
\newblock Long range dependence.
\newblock {\em Found. Trends Stoch. Syst.}, 1(3):163--257.

\bibitem[Spitzer, 1964]{spitzer64principles}
Spitzer, F. (1964).
\newblock {\em Principles of random walk}.
\newblock The University Series in Higher Mathematics. D. Van Nostrand Co.,
  Inc., Princeton, N.J.-Toronto-London.

\bibitem[Taqqu, 1975]{taqqu75weak}
Taqqu, M.~S. (1975).
\newblock Weak convergence to fractional {B}rownian motion and to the
  {R}osenblatt process.
\newblock {\em Z. Wahrscheinlichkeitstheorie und Verw. Gebiete}, 31:287--302.

\bibitem[van~der Vaart and Wellner, 1996]{VaaWel96}
van~der Vaart, A.~W. and Wellner, J.~A. (1996).
\newblock {\em Weak convergence and empirical processes}.
\newblock Springer Series in Statistics. Springer-Verlag, New York.
\newblock With applications to statistics.

\end{thebibliography}
\end{document}